\documentclass[10pt]{article}

\usepackage[english]{babel}
\usepackage[utf8x]{inputenc}
\usepackage[T1]{fontenc}
\usepackage[letterpaper,top=1in,bottom=1in,left=1in,right=1in,marginparwidth=1.75cm]{geometry}
\usepackage{dsfont}
\usepackage{amsmath, amssymb, amsthm, verbatim,enumerate,bbm}
\usepackage{indentfirst, bbm}
\usepackage{appendix}
\usepackage{enumerate}
\usepackage{verbatim}
\usepackage[colorinlistoftodos]{todonotes}
\usepackage[colorlinks=true, allcolors=blue]{hyperref}
\usepackage[nameinlink,capitalise,noabbrev]{cleveref}
\crefname{appsec}{Appendix}{Appendices}
\creflabelformat{enumi}{#2textup{#1}#3}

\title{\vspace{-0.9cm} Number of 1-factorizations of regular high-degree graphs}

\author{Asaf Ferber \thanks{Massachusetts Institute of Technology. Department of Mathematics. Email: {\tt ferbera@mit.edu}. Research is partially supported by an NSF grant 6935855.} \and Vishesh Jain\thanks{Massachusetts Institute of Technology. Department of Mathematics. Email: {\tt visheshj@mit.edu}}\and Benny Sudakov \thanks{Department of Mathematics, ETH, 8092 Zurich, Switzerland. Email: {\tt  benjamin.sudakov@math.ethz.ch.}	Research supported in part by SNSF grant 200021-175573.}}

\date{\today}
\allowdisplaybreaks

\newtheorem{theorem}{Theorem}[section]
\newtheorem*{namedtheorem}{\theoremname}
\newcommand{\theoremname}{testing}

\newtheorem{lemma}[theorem]{Lemma}

\newtheorem{proposition}[theorem]{Proposition}

\newtheorem*{question*}{Question}

\theoremstyle{definition}
\newtheorem{definition}[theorem]{Definition}

\newtheorem{remark}[theorem]{Remark}

\theoremstyle{plain}

\usepackage{algpseudocode,algorithm,algorithmicx}

\newcommand{\Bin}{\ensuremath{\textrm{Bin}}}

\global\long\def\E{\mathbb{E}}

\global\long\def\Pr{\text{Pr}}

\date{}
\begin{document}
\maketitle

\begin{abstract}
A $1$-factor in an $n$-vertex graph $G$  is a collection of $\frac{n}{2}$ vertex-disjoint edges and a $1$-factorization of $G$ is a partition of its edges into edge-disjoint $1$-factors. Clearly, a $1$-factorization of $G$ cannot exist unless $n$ is even and $G$ is regular (that is, all vertices are of the same degree). 
The problem of finding $1$-factorizations in graphs goes back to a paper of Kirkman in 1847 and has been extensively studied since then.
Deciding whether a graph has a $1$-factorization is usually a very difficult question.  For example, it took more than 60 years and an impressive tour de force of
Csaba, K\"uhn, Lo, Osthus and Treglown to prove an old conjecture of Dirac from the 1950s, which says that every $d$-regular graph on $n$ vertices contains a $1$-factorization, provided that $n$ is even and $d\geq 2\lceil \frac{n}{4}\rceil-1$. In this paper we address the natural question of estimating $F(n,d)$, the number of $1$-factorizations in $d$-regular graphs on an even number of vertices, provided that $d\geq \frac{n}{2}+\varepsilon n$. Improving upon a recent result of Ferber and Jain, which itself improved upon a result of Cameron from the 1970s, we show that 
$F(n,d)\geq \left((1+o(1))\frac{d}{e^2}\right)^{nd/2}$, which is asymptotically best possible. 
\end{abstract}
\section{Introduction}

A $1$-factorization of a graph $G$ is a collection $\mathcal M$ of edge-disjoint perfect matchings (also referred to as $1$-factors) whose union is $E(G)$. An equivalent definition of a $1$-factorization is an edge-coloring of $G$ where each color class consists of a perfect matching. Clearly, if $G$ admits a $1$-factorization then the number of vertices of $G$, denoted by $|V(G)|$, is even, and $G$ is a \emph{regular} graph (that is, all its vertices are of the same degree). The problem of finding $1$-factorizations in graphs goes back to a paper of Kirkman \cite{Kirk} from 1847 and has been extensively studied since then in graph theory and in
combinatorial designs (see, e.g., \cite{MR,Wallis} and the references therein). Despite the fact that $1$-factorizations of the complete graph are quite easy to construct (for example, see \cite{Lucas}), the problems of enumerating all the distinct $1$-factorizations and finding $1$-factorizations in graphs which are not complete are considered as much harder. 

A simple example is the problem of finding (and asymptotically enumerating) $1$-factorizations of $K_{n,n}$, the complete bipartite graph with both parts of size $n$. Note that a $1$-factorization of $K_{n,n}$ is equivalent to a Latin square, where a Latin square is an $n\times n$ array, with each row and each column being a permutation of $\{1,\ldots,n\}$ (in particular, each element appears exactly once in each row and each column). The existence of Latin squares follows easily from Hall's marriage theorem.
On the other hand, in order to prove an asymptotic formula for the number of Latin squares, one needs sophisticated estimates of the permanent of the adjacency matrix of regular bipartite graphs. 
Given an $n\times n$ matrix $M_n$, its permanent is  $\text{Per}(M_n) := \sum_{\sigma \in S_n}\prod_{i=1}^{n} m_{i,\sigma(i)}$. If $M$ is a $\{0,1\}$ matrix, then
it is easy to see that the permanent counts the number of perfect matchings in the bipartite graph with both parts of size $n$, where the $i^{th}$ vertex in the first part is connected to the $j^{th}$ vertex in the second part if and only if $m_{i,j}=1$.  It is well known (the upper bound is due to Bregman \cite{Bregman}, solving the conjecture of Minc, and the lower bound was first obtained by Egorychev \cite{Egorychev} and independently by Falikman \cite{Falikman}, solving the conjecture of Van der Waerden) that if $M$ is an $n\times n$ matrix with all entries either $0$ or $1$ whose row sums and column sums are all $d$, then $\text{Per}(M_n)=\left((1+o(1))\frac{d}{e}\right)^{n}$. 
Therefore, starting with a bipartite, $d$-regular graph $H$ with parts of size $n$, by repeatedly removing perfect matchings from $H$, applying the above bound on the obtained graph, and using Stirling's approximation (that is, $d!\approx \left((1+o(1))\frac{d}{e}\right)^d$), we get that there are 
$\left((1+o(1))\frac{d}{e^2}\right)^{nd}$ 
$1$-factorizations of $H$. 

For non-bipartite (regular) graphs, even deciding whether a single 1-factorization exists is usually
a very difficult question.  For example, it took about $60$ years and an impressive tour de force of Csaba, K\"uhn, Lo, Osthus and Treglown \cite{CKLOT} (improving an earlier asymptotic result of Perkovic and Reed \cite{PR}) to solve the following old problem of Dirac:

\begin{theorem}[\cite{CKLOT}]
Every $d$-regular graph $G$ on $n$ vertices, where $n$ is a sufficiently large even integer and $d\geq 2\lceil \frac{n}{4}\rceil-1$, contains a $1$-factorization.
\end{theorem}

The above theorem is clearly tight in terms of $d$ as can be seen, for example if $n=4k+2$, by taking $G$ to be the disjoint union of two cliques of size $2k+1=2\lceil \frac{n}{4}\rceil-1$ (which is odd). 

Once an existence result is obtained, one can naturally ask for the number of distinct such structures. Given a $d$-regular graph $G$, it was shown by Kahn-Lov\'asz (unpublished) and Alon-Friedland \cite{AF} that it has at most $((1+o(1))d/e)^{n/2}$ perfect matchings. Therefore, the same reasoning as above (see also \cite{LL})
shows that the number of $1$-factorizations of $G$ is at most
$\left((1+o(1))\frac{d}{e^2}\right)^{dn/2}$.
On the other hand, no matching lower bounds were known for this problem. For the  complete graph $K_n$, Cameron \cite{Cameron} proved in 1976 that  the number of $1$-factorizations is at least
$\left((1+o(1))\frac{n}{4e^2}\right)^{n^2/2}$ (off by a factor of roughly $4^{-n^2/2}$ from the upper bound), which was recently improved by Ferber and Jain \cite{FJ} to $\left((1+o(1))\frac{n}{2e^2}\right)^{n^2/2}$. For general $d$-regular graphs with $d\geq n/2+\varepsilon n$ only weaker non-trivial lower bounds of the form $n^{(1-o(1))dn/2}$ are proven implicitly in \cite{GLS} and in \cite{FLS}. 

In this paper, we give an asymptotically optimal lower bound for every $d$-regular graph $G$ on $n$ vertices with $d\geq \frac{n}{2}+\varepsilon n$. That is, we prove the following:

\begin{theorem}
  \label{main}
There exists a universal constant $C > 0$ such that for all sufficiently large even integers $n$ and all $d\geq (1/2+n^{-1/C})n$, 
every $d$-regular graph $G$ on $n$ vertices has at least
  $$\left(\left(1-n^{-1/C}\right)\frac{d}{e^2}\right)^{dn/2}$$ distinct
  1-factorizations.
\end{theorem}

\begin{remark}\label{remark:good error term} We have stated the above theorem in a stronger form ($1-n^{-\alpha}$ instead of $1-o(1)$) with the hope that it might be useful in studying the behavior of typical $1$-factorizations. That such a bound might be helpful  was recently shown by Kwan \cite{Kwan} in the study of typical Steiner triple systems (which we do not define here).\end{remark}

We conclude this introduction with a brief outline of the proof of our main result. \\


\noindent
\textbf{Proof outline}: Our proof is based on and extends ideas developed in \cite{FLS} and \cite{FJ}, and largely goes as follows:  First, we find an $r$-regular subgraph $H\subseteq G$, where $r=d^{1-\tau}$, such that any $\Delta$-regular graph $R\supset H$ with $\Delta=(1+o(1))r$ contains a $1$-factorization (we chose the letter $R$ to denote the \emph{remainder} graph obtained after deleting an `approximate' $1$-factorization from $G$) . This is actually the key part of our argument and the existence of such a graph (which, perhaps surprisingly, is quite simple!) is proven in \cref{sec:completion}. 

Next, we show that the graph $G'$ which is obtained by deleting all the edges of $H$ from $G$ contains the `correct' number of `almost' $1$-factorizations. By an `almost' $1$-factorization, we mean a collection of edge-disjoint perfect matchings that cover almost all the edges of $G'$. This part is the most technical part of the paper and is based on a suitable partitioning of the edge-set of $G'$ into sparse subgraphs (quite similar to the one in \cite{FLS}) along with a `nibbling' argument from \cite{DGP}. 

Finally, for any given `almost' $1$-factorization of $G'$, by adding all the edges uncovered by this `almost' $1$-factorization to $H$, we obtain a graph $R\supset H$ which is $\Delta$-regular with $\Delta\approx r$, and therefore admits a $1$-factorization by the property of $H$ discussed above. Since the complete proofs are anyway not too long, we postpone the more formal details to later sections. 

\section{Auxiliary results}

In this section, we have collected a number of tools and auxiliary results to be used in proving our main theorem.

\subsection{Probabilistic tools}
	
Throughout the paper, we will make extensive use of the following well-known bounds on the
	upper and lower tails of the Binomial distribution due to Chernoff
	(see, e.g., Appendix A in \cite{AlonSpencer}).

		\begin{lemma}[Chernoff's inequality]
		Let $X \sim Bin(n, p)$ and let
		${\mathbb E}(X) = \mu$. Then
			\begin{itemize}
				\item
				$\Pr[X < (1 - a)\mu ] < e^{-a^2\mu /2}$
 				for every $a > 0$;
				\item $\Pr[X > (1 + a)\mu ] <
				e^{-a^2\mu /3}$ for every $0 < a < 3/2$.
			\end{itemize}
		\end{lemma}

Sometimes, we will find it more convenient to use the following concentration inequality due to Hoeffding (\cite{Hoeff}).
\begin{lemma}[Hoeffding's inequality]
\label{Hoeffding}
Let $X_{1},\dots,X_{n}$ be independent random variables such that
$a_{i}\leq X_{i}\leq b_{i}$ with probability one. If $S_{n}=\sum_{i=1}^{n}X_{i}$,
then for all $t>0$,
\[
\Pr\left(S_{n}-\E[S_{n}]\geq t\right)\leq\exp\left(-\frac{2t^{2}}{\sum_{i=1}^{n}(b_{i}-a_{i})^{2}}\right)
\]
and 
\[
\Pr\left(S_{n}-\E[S_{n}]\leq-t\right)\leq\exp\left(-\frac{2t^{2}}{\sum_{i=1}^{n}(b_{i}-a_{i})^{2}}\right).
\]
\end{lemma} 

\noindent

\subsection{Completion} 
\label{sec:completion}
In this section we present the completion step, which uses some ideas from \cite{FJ}, and is a key ingredient of our proof. Before stating the relevant lemma, we need the following definition.
\begin{definition}
\label{defn-good-graph}
A graph $H=(A\cup B,E)$ is called $(\alpha,r,m)$-good if it
satisfies the following properties:
\begin{enumerate}[$(G1)$]
\item $H$ is an $r$-regular, balanced bipartite graph with $|A|=|B|=m$.
\item Every balanced bipartite subgraph $H'=(A'\cup B',E')$ of $H$ with
$|A'|=|B'|\geq(1-\alpha)m$ and with $\delta(H')\geq(1-2\alpha)r$
contains a perfect matching.
\end{enumerate}
\end{definition}
The motivation for this definition comes from the next proposition, which shows that a regular graph on an even number of vertices, which can be decomposed into a good graph and a graph of `small' maximum degree, has a 1-factorization.
\begin{proposition}
\label{proposition-completion}
There exists a sufficiently large integer $m_0$ for which the following holds. Let $m\geq m_0$, and 
suppose that $H=(A\cup B, E(H))$ is an $(\alpha,r_1,m)$-good graph with $m^{1/10} \leq r_1 \leq m$ and $\log{m}/r_1 \ll \alpha^{2} < 1/100$. Then, for every $r_2\leq \alpha^4r_1/\log{m}$, every $r:=r_1 + r_2$-regular (not necessarily bipartite) graph $R$ on the vertex set $A\cup B$, for which $H\subseteq R$, admits a 1-factorization.
\end{proposition}
\begin{proof}
First, observe that $e(R[A])=e(R[B])$. Indeed, as $R$ is $r$-regular, we have for $X\in \{A,B\}$ that 
$$rm=\sum_{v\in X}d_R(v)=2e(R[X])+e(R[A,B]),$$ from which the above equality follows. Moreover, $\Delta(R[X])\leq r_2$ for all $X\in \{A,B\}$ since the only edges in $R[X]$ come from $R\setminus H$. 
Next, let $R_0:=R$ and $f_0:=e(R_0[A])=e(R_0[B])$.  
By Vizing's theorem (\cite{Vizing}), both $R_0[A]$ and $R_0[B]$ contain matchings of size exactly $\lceil f_0/(r_2+1)\rceil$. Consider any two such matchings $M_A$ in $A$ and $M_B$ in $B$, and for $X\in \{A,B\}$, let $M'_X\subseteq M_X$ 
denote a matching of size $|M'_X| =\lfloor \alpha f_0/2r_2 \rfloor$ such that no vertex $v \in V(H)$ is adjacent to more than $3\alpha r_1/2$ vertices which are paired in the union of the two matchings. 

To show that such $M'_X$ must exist, it suffices to show that for $X\in \{A,B\}$, there exist matchings $M''_{X}\subseteq M_X$ with $|M''_X| \geq \lfloor \alpha f_0/2r_2 \rfloor$ such that no vertex $v\in V(H)$ is incident in $H$ to more than $3\alpha r_1/2$ vertices which are paired in the union of the two matchings, since we can obtain $M'_X$ from $M''_X$ by simply removing the appropriate number of edges arbitrarily. Note that if the size of $M_X$ is at most $3\alpha r_1/4$, then this follows trivially by taking $M''_X = M_X$. When the size of $M_X$ is at least $3\alpha r_1/4$, the existence of $M''_X$ is seen using the following simple probabilistic argument. Let $M''_X$ denote the random matching obtained by including each edge of $M_X$ independently with probability $3\alpha/4$. By Chernoff's inequality, $|M''_X| \geq 3\alpha |M_X|/5 > \lfloor \alpha f_0/2r_2\rfloor$, except with probability at most $\exp(-\Theta(\alpha |M_X|)) \leq \exp(-\Theta(\alpha^{2}r_1))\ll 1$, where the last inequality uses the assumption that $\alpha^{2} r_1 \gg \log{m} \gg 1$. Moreover, each $v\in V(H)$ is incident to at most $r$ vertices, each of which has at most $1$ edge in $M_A \cup M_B$. Hence, the number of vertices paired in $M''_A \cup M''_B$ incident to a fixed $v\in V(H)$ is stochastically dominated by $Bin(r,3\alpha /4)$. In particular, by Chernoff's inequality, the probability that a fixed $v\in V(H)$ is incident to at least $3\alpha r_1/2$ vertices matched in $M''_A \cup M''_B$ is at most $\exp(-\Theta(\alpha r)) \ll 1/m$, where the last inequality uses the assumption that $\alpha r_1 \geq \alpha^{2}r_1 \gg \log{m}$. Therefore, the union bound shows that the above claim holds simultaneously for all vertices $v\in V(H)$ with probability close to $1$. It follows that there exist $M''_A$ and $M''_B$ with the desired properties.       


Delete the \emph{vertices} in $\left(\cup M'_A\right)\bigcup \left(\cup M'_B\right)$, as well as any edges incident to them, from  $H$ and denote the resulting graph by $H'=(A'\cup B',E')$. Since $|A'|=|B'|\geq (1-\alpha)|A|$ and $\delta(H')\geq (1-3\alpha/2)r_1$ by the choice of $M'_X$, it follows from $(G2)$ that $H'$ contains a perfect matching $M'$. Note that $M_0:=M'\cup M'_A\cup M'_B$ is a perfect matching in $R_0$. We repeat this process with $R_1:=R_0 - M_0$ (deleting only the edges in $M_0$, and not the vertices) and $f_1 :=e(R_1[A])=e(R_1[B])$ until we reach $R_k$ and $f_k$ such that $f_k\leq r_2$. Since $f_{i+1} \leq \left(1-\alpha/3r_2\right)f_i$, this must happen after at most $3r_2\log{m}/\alpha < \alpha^{2} r_1$ steps. Moreover, since $\deg(R_{i+1}) = \deg(R_i)-1$, it follows that during the first $\lceil 3r_2\log {m}/\alpha\rceil$ steps of this process, the degree of any $R_j$ is at least $r_1 - \alpha^2 r_1$. Therefore, since $(r_1 - \alpha^2r_1) - 3\alpha r_1/2 \geq (1-2\alpha)r_1$, we can indeed use $(G2)$ throughout the process, as done above.  
 
From this point onwards, we continue the above process (starting with $R_k$) with matchings of size one i.e. single edges from each part, until no more edges are left. By the choice of $f_k$, we need at most $r_2$ such iterations, which is certainly possible since $r_2 + 3r_2\log{m}/\alpha < \alpha^{2} r_1$ and $(r_1 - \alpha^{2}r_1)-3\alpha r_1/2 \geq (1-2\alpha)r_1$. After removing all the perfect matchings obtained via this procedure, we are left with a regular, balanced, \emph{bipartite} graph, which admits a 1-factorization (this follows by a more or less direct application of Hall's marriage theorem \cite{Hall}). Taking any such 1-factorization along with all the perfect matchings that we removed gives a 1-factorization of $R$.     
\end{proof}

The remainder of this subsection is devoted to proving the following proposition, which shows that every $d$-regular graph on $n$ vertices contains a `good' subgraph $H$, assuming that $d\geq n/2+\varepsilon n$ and $n$ is a sufficiently large even integer.

\begin{proposition}
  \label{lemma: G contains a resilient subgraph}
Let $n$ be a sufficiently large even integer, and let $0 < \varepsilon = \varepsilon(n) < 1$ be such that $\varepsilon^{4}\gg \log{n}/n$. Let $G$ be a graph on $n$ vertices  which is $d$-regular, with $d\geq n/2+\varepsilon n$. Then, for every $p=\omega\left( \frac{\log n}{n\varepsilon^{3}}\right)\leq 1$, there exists a spanning subgraph $H$ of $G$ which is $(\varepsilon/10,r,n/2)$-good, for some $(1-\varepsilon/1000)dp/2 \geq r\geq (1-\varepsilon/100)dp/2$.
\end{proposition}

To prove this proposition, we will use the following three results. The first result is a theorem from \cite{FKS}, which states that if $G$ is a bipartite graph with sufficiently large minimum degree which contains an $r$-factor (i.e. a spanning $r$-regular subgraph) for $r$ sufficiently large, then the random graph $G_p$, which is obtained by keeping each edge of $G$ independently with probability $p$, typically contains a $(1-o(1))rp$-factor. The proof of this theorem follows quite easily by using the Gale-Ryser criterion for the existence of $r$-factors in bipartite graphs (\cite{Gal}, \cite{Ry}), and standard applications of Chernoff's bounds.

\begin{theorem}[Theorem 1.4 in \cite{FKS}]
\label{thm:random contains factor}
Let $m$ be a sufficiently large integer. Then, for any positive $\tau$ such that $\log{m}/m \ll \tau^{3} < 1$, $\alpha = 1/2+\tau$ and $0 < \rho \leq \alpha $, the following holds. Suppose that: 
\begin{enumerate}
  \item $G$ is bipartite with parts $A$ and $B$, both of size $m$,
  \item $\delta(G)\geq \alpha m$, and
  \item $G$ contains a $\rho m$-factor.
\end{enumerate}
Then, for $p=\omega\left(\frac{\log m}{m\tau^{3}}\right)$, the random graph $G_p$ has a $k$-factor for $k=(1-\tau)\rho mp$ with probability $1-m^{-\omega(1)}$.
\end{theorem}

\begin{remark}
In \cite{FKS}, $\tau$ is taken to be some positive constant, as opposed to a function of $m$ which can go to $0$ as $m$ goes to infinity. However, the exact same proof actually gives the slightly more general result stated above.     
\end{remark}

The second result shows that if $G$ is a bipartite graph with parts of size $m$, then with high probability, the number of edges in $G_p$ between subsets $X$ and $Y$ with $|X|=|Y|\leq m/2$ is not much more than $pm|X|/2$.  
\begin{lemma}
\label{lemma:expandingneighbourhood}
Let $G = (A\cup B,E)$ be a bipartite graph with parts $A$ and $B$, both of size $m$. Let $c = 1/2 + \tau$ with $0<\tau = \tau(m) < 1$, and let $p=\omega\left(\frac{\log{m}}{m\tau^{2}}\right)$. Then, for $G_p$, the following holds with probability at least $1-m^{-\omega(1)}$: $e_{G_p}(X,Y) < cpm|X|$ for any subsets $X\subseteq A$ and $Y\subseteq B$ with $|X|=|Y| \leq m/2$.
\end{lemma}
\begin{proof}
Consider any subsets $X\subseteq A$ and $Y\subseteq B$ with $|X|=|Y| \leq m/2$. Since
$$e_G(X,Y) \leq |X||Y| \leq |X|\frac{m}{2},$$
we get that
$$\Pr\left[e_{G_{p}}(X,Y)\geq cpm|X|\right]  \leq\Pr\left[\Bin\left(\frac{m|X|}{2},p\right)\geq cpm|X|\right]
 \leq\exp(-\tau^{2}m|X|p/6),$$
where the first inequality follows from the fact that $e_{G_p}(X,Y)$ is a sum of at most $m|X|/2$ independent $\text{Bernoulli}(p)$ random variables, and the second inequality follows from Chernoff's bounds.

Let $B$ denote the event that there exist subsets $X\subseteq A$ and $Y \subseteq B$ with $|X|=|Y| \leq m/2$ and $e_{G_p}(X,Y) \geq cpm|X|$. Then, it follows by the union bound that 
\begin{align*}
\Pr[B] & \leq\sum_{x=1}^{m/2}{m \choose x}^{2}\exp(-\tau^{2}mxp/6)\leq\sum_{x=1}^{m/2}\left(\frac{em}{x}\right)^{2x}\exp(-\tau^{2}mxp/6) 
\leq\sum_{x=1}^{m/2}\exp(4x\log m-\tau^{2}mxp/6)\\
&\leq \sum_{x=1}^{m/2}\exp\left(x\log{m}(4-\omega(1))\right)
= m^{-\omega(1)},
\end{align*}
where the inequality on the second line holds since $pm = \omega(\log{m}/\tau^{2})$.
\end{proof}

Finally, the third result, which is a lemma from \cite{FLS}, shows that an almost regular bipartite graph with sufficiently large degrees contains an $r$-factor with $r$ close to its minimum degree. 

\begin{lemma}[Lemma 20 in \cite{FLS}]\label{lemma: almost regular contains regular}
Let $\rho \geq 1/2$, $m \in \mathbb{N}$, and $\xi = \xi(m) >0$. Suppose that $G=(A\cup B,E)$ is a bipartite graph with parts $A$ and $B$, both of size $m$, and $\rho m+\xi\leq \delta(G)\leq \Delta(G)\leq \rho m+\xi +\xi^2/m$. Then, $G$ contains a $\rho m$-factor.
\end{lemma}

We are now ready to prove \cref{lemma: G contains a resilient subgraph}.
\begin{proof}[Proof of \cref{lemma: G contains a resilient subgraph}]
Consider a random partitioning of $V(G)$ with parts $A,B$ of size $m:=|A|=|B|=n/2$, and let $G'=(A\cup B,E')$ denote the induced bipartite subgraph between $A$ and $B$. For any $v\in V(G)$, by linearity of expectation 
$$\E[\deg_{G'}(v)] = \frac{d}{2} + \frac{d}{2(n-1)}.$$
Since the distribution of the random subset $A\subseteq V(G)$ coincides with the distribution on subsets of $V(G)$ obtained by including every vertex independently with probability $1/2$, conditioned on the event (of probability $\Theta(1/\sqrt{n})$) that exactly $n/2$ elements are included, it follows from Hoeffding's inequality that for any fixed vertex $v\in V(G)$ the probability that $|\deg_{G'}(v) - \E[\deg_{G'}(v)]| \geq 4\sqrt{d\log{n}}$ is at most $O(\sqrt{n}\exp(-16d\log{n}/d))\ll 1/n$. Therefore, taking the union bound over all vertices, it follows that with high probability,    
$$d/2-5\sqrt{d\log n}\leq \delta(G')\leq \Delta(G')\leq d/2+5\sqrt{d\log n}.$$
Fix any such $G'$. 

We now wish to apply \cref{lemma: almost regular contains regular} to $G'$. For the choice of parameters, note that since $d \geq n/2 + \varepsilon n$, we can find $\rho \geq 1/2$ such that $\rho m = d/2 - \varepsilon m/1000$. If we take $\xi = \varepsilon m/1000 - 5\sqrt{d\log{n}}$, then since $\varepsilon^{4} \gg \log{n}/n$, we have $\xi \geq \varepsilon m/2000$ for all $n$ sufficiently large, and hence $\xi^{2}/m \geq O(\varepsilon^{2} m) \gg \sqrt{n\log{n}}$. This shows that 
$$\rho m+\xi\leq \delta(G')\leq \Delta(G')\leq \rho m+\xi+\xi^2/m,$$
Therefore, \cref{lemma: almost regular contains regular} applied  to $G'$ shows that $G'$ contains a $\rho m=d/2 - \varepsilon m/1000$-factor. Note also that $\delta(G')\geq \alpha m$, with $\alpha = \rho + (\xi/m) \geq 1/2 + \tau$, where $\tau = \varepsilon/1000$. 

We will now apply \cref{thm:random contains factor} and \cref{lemma:expandingneighbourhood} to $G'$ in order to extract a sparse `good' subgraph. 
Let $p=\omega(\log m/(m\varepsilon^{3}))$ and consider the random graph $G'_p$. Since $G'$ satisfies the hypothesis of \cref{thm:random contains factor} with the parameters $\rho$, $\alpha$ and $\tau$ as above, it follows that with high probability, $G'_p$ contains an $r=(1-\tau)\rho mp = (1-\tau)(1-\varepsilon m /500d)dp/2$-factor $H$. In particular, note that $(1-\varepsilon/100)dp/2\leq (1-\tau)(1-\varepsilon /500)dp/2\leq r\leq (1-\tau)dp/2 = (1-\varepsilon/1000)dp/2$ as desired. Moreover, applying \cref{lemma:expandingneighbourhood} to $G'$ with $p$ and $\tau$ as above, we see that with high probability, for any subsets $X\subseteq A$ and $Y\subseteq B$, $e_{G'_p}(X,Y)< (1/2 + \varepsilon/1000)pm|X|$. Fix any such realization of $G'_p$. We will show that $H$ is $(\varepsilon/10,r,n/2)$-good.   


Suppose this is not the case. Then, by definition, there must exist a balanced bipartite subgraph $H'=(A'\cup B',E')\subseteq H\subseteq G'_{p}$, with parts $A',B'$ of size $|A'|=|B'|\geq (1-\varepsilon/10)n/2$ and with $\delta(H') \geq (1-\varepsilon/5)r$, which does not contain a perfect matching. Therefore, by Hall's marriage theorem (see, e.g., \cite{SS} for the version used here), there must exist subsets $X'\subseteq A'$ and $Y'\subseteq B'$ with $|X'|=|Y'|\leq {n/4}$ such that at least one of the following is true: $$N_{H'}(X')\subseteq Y'\text{, or } N_{H'}(Y')\subseteq X'.$$ In either case, we get from the minimum degree assumption on $H'$ that $$(1-\varepsilon/5)r|X'| \leq e_{H'}(X',Y') \leq e_{G'_p}(X',Y').$$
Thus, we get \begin{align*}
e_{G'_{p}}(X',Y') & \geq\left(1-\frac{\varepsilon}{5}\right)r|X'| \geq\left(1-\frac{\varepsilon}{5}\right)\left(1-\frac{\varepsilon}{100}\right)\frac{dp}{2}|X'| \geq\left(1-\frac{\varepsilon}{4}\right)\left(\frac{1}{2}+\varepsilon\right)\frac{np}{2}|X'|\\
&\geq\left(\frac{1}{2}+\frac{3\varepsilon}{4}\right)\frac{np}{2}|X'|.
\end{align*}
However, since $G'_p$ satisfies the conclusion of \cref{lemma:expandingneighbourhood}, we must also have $e_{G'_p}(X',Y') \leq (1/2 + \varepsilon/1000)pm|X'| = (1/2 + \varepsilon/1000)np|X'|/2$, which leads to a contradiction.     
\end{proof}

\subsection{Random partitioning}

The following technical lemma allows us to partition our graph $G$ into a number of smaller subgraphs, each of which contains many `almost' $1$-factorizations. 
Its proof is similar to Lemma 27 in \cite{FLS}, but we need here a different set of parameters.   

\begin{lemma}\label{partition lemma}
Let $n$ be a sufficiently large integer and let $K$ be an integer in $[\log^2{n}, n^{1/300}]$. Let $\tau > 0$ be such that $\tau > 100/K$. 
    Suppose that $G$ is a $d$-regular graph on $n$ vertices with
    $d \geq n/2$. 
    Then, there
	are $K^3$ edge-disjoint spanning subgraphs
	$H_1,\ldots ,H_{K^3}$ of $G$ with the following properties:
		\begin{enumerate}
			\item For each $H_i$, there is a partition $V(G) = U_i \cup W_i$ with
			$|W_i|  =  n/K^2 \pm (n/K^2)^{2/3}$ and even;
            			\item Letting $F_i = H_i[W_i]$, we have $\delta (F_i)
			\geq (d/n-\tau)|W_i| $;
			\item Letting $E_i = H_i[U_i, W_i]$, we have
			$e _{E_i} (u,W_i) \geq |W_{i}|/10K^3$ for all $u \in U_i$;
\item Letting $D_i = H_i[U_i]$ for all $i$, then for some
			$d/K^{3} \geq r \geq (1-\tau )d/K^3$, we have
			$$r \leq  \delta(D_i)
			\leq \Delta(D_i)\leq r +r^{4/5}.$$

		\end{enumerate}
\end{lemma}

\begin{proof}
First, let $\{S(v)\}_{v\in V(G)}$ be i.i.d random variables, where for each $v\in V(G)$, $S(v)\subseteq [K^3]$ is a subset of size exactly $K$, chosen uniformly at random from among all such subsets. For each $i \in [K^3]$, let $$W_i := \{v\in V(G) : i\in S(v)\};$$
note that the graphs $G[W_i]$ are not necessarily edge-disjoint. 

Second, let $s:= n/K^2$. Since any $i\in [K^{3}]$ is included in $S(v)$ with probability $1/K^{2}$, independently for different $v\in V(G)$, it follows by Chernoff's bounds that the following holds for all $v \in V(G)$ and $i \in [K^3]$ with probability $1-o(1)$: 
\begin{enumerate}[$(a)$]
\item $|W_i| = s(1\pm s^{-1/3})$; 
\item $e_G(v, W_i) = \frac{d}{K^2}(1\pm s^{-1/3})$. 
\end{enumerate}
Next, for each $v \in V(G)$, we define the random variable $Y(v)$ to be the number of vertices $u \in N_G(v)$ with $\{u,v\} \subset W_i$ for some $i\in [K^3]$. For each $v \in V(G)$ and $i \in [K^3]$, we define the random variable $Z_i(v)$ to be the number of vertices $u \in N_G(v)$ such that $u \in W_i$ and $\{u,v\} \subset W_j$ for some $j\in [K^3]$. Since all vertices of $G$ have the same degree, the values of $\mathbb{E}[Y(v)]$ and $\E[Z_i(v)]$ are the same for all choices of $v$ and $i$. Let us denote these common values by $Y$ and $Z$, respectively. 

We claim that $Y \leq d/K$ and $d/2K^{3} \leq Z \leq 2d/K^3$ for all sufficiently large $K$. To see this, note that $Y=\E[Y(v)] = \E[\E[Y(v)|S(v)]]$, and conditioned on any realization of $S(v)$, 
$$\E[Y(v)|S(v)] \leq \sum_{u\in N_{G}(v)}\E_{S(u)}[|S(u)\cap S(v)|] = d\E_{S(u)}[|S(u)\cap S(v)|] = d/K.$$ 
Similarly, $Z = \E[Z_1(v)] = \E[\E[Z_1(v)|S(v)]]$. Conditioned on any realization of $S(v)$ for which $1\notin S(v)$, for all sufficiently large $K$,
$$\E[Z_1(v)|S(v)] = \sum_{u\in N_{G}(v)}\Pr_{S(u)}[\{|S(u)\cap S(v)| > 0 \}\cap \{1\in S(u)\}] \in [0.9d/K^{3}, 1.1d/K^{3}],$$
whereas conditioned on any realization of $S(v)$ for which $1 \in S(v)$, 
$$\E[Z_1(v)|S(v)] = d\Pr_{S(u)}[\{|S(u)\cap S(v)| > 0\} \cap \{1\in S(u)\}] = d/K^{2}.$$ 
Since $\Pr[1\in S(v)] = 1/K^{2}$, the desired conclusion follows from the law of total probability. 
Also, by Hoeffding's inequality, it follows that with probability $1-o(1)$, for all $v \in V(G)$ and $i \in [K^3]$, 
\begin{enumerate}[$(a)$]
\setcounter{enumi}{2}
\item $Y(v) = Y \pm 2\sqrt{n \log{n}}$ ;
\item $Z_i(v) = Z \pm 2\sqrt{n \log{n}}$.
\end{enumerate} 
Therefore, there exists a collection $W_1,\dots, W_{K^3}$ satisfying $(a),(b),(c)$ and $(d)$ simultaneously. Moreover, after removing at most one vertex from each $W_i$, we may further assume that $|W_i|$ is even for all $i \in [K^3]$.  Fix any such collection. Let $U_i:= V(G)\setminus W_i$ and 
let $G'=(V(G),E(G'))$, where $E(G'):=E(G)\setminus \cup_{i\in[K^3]} E[W_i]$.  

To each edge $e\in \cup_{i\in[K^3]}E[W_i]$, assign an arbitrary $k(e) \in [K^3]$ such that $e \subset W_{k(e)}$. Further, assign independently to each edge $e \in E(G')$, a uniformly chosen element $k(e) \in [K^3]$. For each $i\in [K^3]$, let $H_i=(V(G), E(H_i))$, where $E(H_i):=\{e\in E(G) : k(e)=i\}$. We claim that with probability at least $1 - 2n^{-8}$, $H_1,\dots,H_{K^3}$ satisfy the conclusions of the lemma. 

Conclusion $1.$ follows immediately from property $(a)$. For conclusion $2.$, note that the only edges $\{u,v\}\subset W_i$ which are present in $G[W_i]$ but possibly not in $H_i[W_i]$ are those which are also present in $G[W_j]$ for some other $j\in [K^{3}]$. Since for given $i\in [K^3]$ and $v\in V(G)$, $Z_i(v)$ bounds the number of such edges incident to $v$, it follows from properties $(b)$ and $(d)$ that for any $i \in [K^3]$ and any $v\in W_i$:
$$\deg_{F_i}(v) \geq e_G(v, W_i) - Z_i(v)\geq  e_G(v,W_i) - \frac{3d}{K^3} \geq \frac{d}{K^2} - \frac{4d}{K^3}.$$ 
Therefore, property $(a)$ shows that $$\deg_{F_i}(v) \geq \frac{d}{n}|W_i| - \frac{4d}{K^3} - \frac{d}{n}\frac{n^{2/3}}{K^{4/3}} \geq \left(\frac{d}{n} - \frac{5}{K}\right)|W_i|.$$

We now verify that conclusions $3.$ and $4.$ are satisfied with the desired probability. Property $(c)$ shows that for all $v \in V(G)$,  
$$\deg_{G'}(v) = \deg_G(v) - Y(v) = d-Y \pm 2\sqrt{n \log{n}}.$$ 
Moreover, properties $(b)$ and $(d)$ show that for all $i\in [K^3]$ and for all $u \in U_i$, 
$$e_{G'}(u,W_i) = e_G(u,W_i) - Z_i(u) = \frac{d}{K^2}(1\pm s^{-1/3}) - Z \pm 2\sqrt{n \log{n}}.$$
Therefore, since each edge $e$ in $G'$ chooses a label $k(e)\in [K^3]$ independently and uniformly, it follows by Chernoff's inequality that for all $v \in V(G)$, $i \in [K^3]$ and $u \in U_i$, 
\begin{align*}
\deg_{H_{i}}(v) & =\frac{\deg_{G'}(v)}{K^{3}}\pm2\sqrt{n\log n}= \frac{d-Y}{K^3} \pm 4\sqrt{n\log{n}}\\
e_{H_{i}}(u,W_{i}) & =\frac{e_{G'}(u,W_{i})}{K^{3}}\pm2\sqrt{n\log n}= \frac{d}{K^5}\left(1-\frac{ZK^2}{d}\right) \pm n^{2/3},
\end{align*}
except with probability at most (say) $2n^{-8}$. Whenever this holds, we also get that for all $i\in [K^3]$ and $u \in U_i$, 
$$ \deg_{D_i}(u) = \deg_{H_i}(u) - e_{H_i}(u,W_i) = \frac{d-Y}{K^3} - \frac{d}{K^5}\left(1-\frac{ZK^2}{d}\right) \pm 2n^{2/3}.$$
This implies that for all $i\in [K^3]$ and $u\in U_i$,
\begin{align*}
e_{E_i}(u,W_i) 
& \geq \frac{d}{K^5} - \frac{Z}{K^{3}} - n^{2/3} 
\geq \frac{n}{2K^5} - \frac{2d}{K^6} - n^{2/3}
\geq \frac{n}{2K^5} - \frac{2n}{K^6} - n^{2/3} \geq \frac{|W_i|}{3K^{3}} - \frac{2|W_i|}{K^4} - n^{2/3}\\
& \geq \frac{|W_i|}{10K^3},
\end{align*}
where the last line uses $|W_i|/K^{3} \gg |W_i|/K^{4} \gg n^{2/3}$. We also have for all $i\in [K^3]$ that
\begin{align*}
\delta(D_i) 
&\geq \frac{d}{K^3} - \frac{Y}{K^3} - \frac{d}{K^5} + \frac{Z}{K^3} - 2n^{2/3} \geq \frac{d}{K^3} - \frac{d}{K^4} - \frac{d}{K^5} + \frac{d}{2K^6} - 2n^{2/3}\\
&\geq \left(1-\frac{5}{K}\right)\frac{d}{K^3},
\end{align*}
and
\begin{align*}
\Delta(D_i)
&\leq \frac{d}{K^3} - \frac{Y}{K^3} - \frac{d}{K^5} + \frac{Z}{K^3} + 2n^{2/3} \leq \delta(D_i) + 4n^{2/3}\\
&\leq \delta(D_i) + \delta(D_i)^{3/4}.
\end{align*}
Therefore, we may take $(1-\tau)d/K^{3} \leq r \leq d/K^{3}$ in conclusion 4. This completes the proof. 
\end{proof}

\begin{remark}
\label{rmk:counting partitions}
The proof of \cref{partition lemma} given above actually shows that if we fix any collection $W_1,\dots,W_{K^{3}}$ satisfying properties $(a),(b),(c)$ and $(d)$ in the proof, then there are at least $$\left(1-2n^{-8}\right)(K^{3})^{\left(1-\frac{10}{K}\right)\frac{nd}{2}}$$
collections $H_1,\dots,H_{K^3}$ satisfying the conclusions of the lemma with respect to this choice of $\{U_i,W_i\}_{i=1}^{K^3}$. This may be seen as follows: since we have seen that, given $W_1,\dots,W_{K^3}$ satisfying properties $(a),(b),(c)$ and $(d)$, the random process to produce $H_1,\dots,H_{K^{3}}$ with the desired properties succeeds with probability at least $1-2n^{-8}$, it suffices to show that the number of outcomes of this random process is at least $(K^{3})^{\left(1-\frac{10}{K}\right)\frac{nd}{2}}$. But this is immediate since 
$$|E(G')| \geq |E(G)| - \sum_{i=1}^{K^3}\frac{|W_i|^{2}}{2} \geq \frac{dn}{2} - K^{3}\cdot \frac{n^2}{K^4} \geq \frac{dn}{2}\left(1-\frac{10}{K}\right),$$ 
and each edge $e\in E(G')$ chooses one of $K^3$ labels. The fact that all these collections satisfy the conclusion of the lemma with respect to the \emph{same fixed} choice of $\{U_i,W_i\}_{i=1}^{K^3}$ will be used crucially in the proof of \cref{main}.    
\end{remark}

\subsection{Almost regular graphs contain many equitable collections of edge-disjoint large matchings}
The following proposition shows that an almost regular graph contains the `correct' number of collections of large matchings such that every collection is equitable in the sense that each vertex is left uncovered by only a small number of the matchings. The proof of this proposition follows from the proof of the main result in the work of Dubhashi, Grable, and Panconesi \cite{DGP}. For completeness, we include the details in \cref{appendix:nibbling}, after the proof of our main result.   
\begin{proposition} 
\label{prop:nibbling-count}
Let $n$ be a sufficiently large even integer, and let $G$ be a graph on $n$ vertices and $m$ edges with 
$\delta:=\delta(G)\leq\Delta(G)=:\Delta$
such that $\delta \geq n^{1/10}$ and $\Delta - \delta \leq \Delta^{5/6}$. There exists a universal constant $J > 0$ for which the following holds. There are at least $\left(\frac{(1-n^{-1/40J})\delta}{e^2}\right)^{m}$ distinct collections of edge-disjoint matchings $\mathcal{M}=\{M_1,\dots,M_\delta\}$ of $G$ such that: 
\begin{enumerate}
\item Each matching $M_i$ covers at least $\left(1-n^{-1/10J}\right)n$ vertices; 
\item Each vertex is uncovered by at most $\delta\cdot n^{-1/10J}$ matchings in $\mathcal{M}$.
\end{enumerate}
\end{proposition}

\subsection{Extending edge disjoint matchings to edge disjoint perfect matchings}
In this subsection, we show how to complete a collection of edge-disjoint matchings of $H[U_i]$ into a collection of edge-disjoint perfect matchings of $G$, using the sets $W_i$. 

\begin{lemma}
\label{lemma:extend-to-pm}
Let $n$ be a sufficiently large even integer. Let $K$ be an integer in $[\log^{10}{n}, n^{1/300}]$ and let $H$ be a graph on $n$ vertices for which: 
\begin{enumerate}
\item $V(H)=U\cup W$ with $|W|$ even;
\item $|W|=\frac{n}{K^2} \pm (\frac{n}{K^2})^{2/3}$; 
\item $\delta(H[W])\geq (1/2+\tau/2)|W|$ with $\tau > 100/K$;
\item Every vertex $u\in U$ has at least $\frac{n}{K^6}$ edges into $W$.
\end{enumerate}
Let $\mathcal M$ be a collection of $t\leq \frac{10n}{K^3}$ edge-disjoint matchings of $H[U]$ such that:
\begin{enumerate}[\indent(a)]
\item Every matching in $\mathcal M$ covers at least $|U|-\frac{n}{K^{10}}$ vertices of $U$;
\item Every vertex $u\in U$ is uncovered by at most $n/K^{10}$ matchings in $\mathcal M$. 
\end{enumerate}
Then, $\mathcal M$ can be extended to a collection of $t$ edge-disjoint perfect matchings of $H$.
\end{lemma}

\begin{proof} 
Let $\mathcal M:=\{M_1,\ldots, M_t\}$ be an enumeration of the matchings. For each $M_i$, let $C_i \subseteq U$ denote the set of vertices which are not covered by $M_i$. By assumption (a), we have $|C_i| \leq n/K^{10}$ for all $i \in [t]$. We now describe and analyze an iterative process to extend $M_1,\dots,M_t$ to edge-disjoint perfect matchings $\overline{M_1},\dots, \overline{M_t}$ of $H$. 

Let $H_1:=H$. For each $u \in C_{1}$, select a distinct vertex $w(u) \in W$ such that $\{u,w(u)\}$ is an edge in $H_1$. This is possible (and can be done greedily) since every $u \in C_{1}$ has at least $n/K^{6} > |C_1|$ edges into $W$ by assumption 4. Let $W_1 \subseteq W$ denote the set of vertices in $W$ which have not been matched to any vertex in $C_{1}$. Note that since $|W|$ and $|C_1|$ are even by assumption, $|W_1|$ is also even.  Moreover, $$|W_1| = |W| - |C_1| \geq |W|- \frac{n}{K^{10}} \geq |W|\left(1-\frac{1}{K^7}\right),$$ 
where the last inequality follows from assumption 2. and 
$$\delta(H_1[W_1]) \geq \delta(H_1[W]) - (|W|-|W_1|) \geq \left(\frac{1}{2}+\frac{\tau}{2}\right)|W| - \frac{|W|}{K^{7}} > \frac{1}{2}|W_1|,$$
where the second and third inequalities follow from assumption 3.
A classical theorem due to Dirac shows that any graph on $2k$ vertices with minimum degree at least $k$ contains a Hamilton cycle, and hence, a perfect matching; therefore, $H_1[W_1]$ contains a perfect matching $N_1$. Let $\overline{M_1}:= M_1 \cup N_1 \cup \{\{u,w(u)\}: u\in C_1\}$. It is clear that $\overline{M_1}$ is a perfect matching in $H_1$. Continue this process starting with $H_2$, where $H_2$ is the graph obtained from $H_1$ by deleting the edges of $\overline{M_1}$. 

To complete the proof, it suffices to show that the above procedure can be repeated $t$ times. For this, we simply need to observe two things. First, since each vertex $u \in U$ is uncovered by at most $n/K^{10}$ matchings $M_i$ by assumption (b), we need to use at most $n/K^{10}$ edges from $u$ into $W$ during this process; in particular, at any stage $i\in [t]$ during this process, every $u \in U$ has at least $n/K^{6} - n/K^{10} > |C_i|$ edges into $W$. Second, since $\delta(H_{i+1}[W]) \geq \delta(H_i[W]) - 1$ for all $i \in [t-1]$ and since $t \leq 10n/K^{3} \leq \tau |W|/4$, it follows that at any stage $i \in [t]$, $$\delta(H_i[W]) \geq \left(\frac{1}{2} + \frac{\tau}{2}\right)|W| - t \geq \left(\frac{1}{2} + \frac{\tau}{4}\right)|W|,$$ which is sufficient for the application of Dirac's theorem as above.       
\end{proof}

\begin{remark} 
\label{rmk:distinct-stays-distinct}
The above proof shows that if $\mathcal{M}=\{M_1,\dots,M_t\}$ and $\mathcal{M'}=\{M'_1,\dots,M'_t\}$ are distinct collections of $t$ edge-disjoint matchings in $H[U]$ satisfying the hypotheses of \cref{lemma:extend-to-pm}, then $\overline{\mathcal{M}}=\{\overline{M_1},\dots,\overline{M_t}\}$ and $\overline{\mathcal{M'}}=\{\overline{M'_1},\dots,\overline{M'_t}\}$ are distinct collections of $t$ edge-disjoint perfect matchings in $H$. This is because none of the edges in $\overline{M_i}\setminus M_i$ and $\overline{M'_i} \setminus M'_i$  are present in $H[U]$, so that $\overline{M_i}\cap H[U] = M_i$ and $\overline{M'_i}\cap H[U] = M'_i$.  
\end{remark}

\section{Proof of \cref{main}}

In this section we prove our main result, \cref{main}.

\begin{proof}[Proof of \cref{main}]
Let $C = 2000\max\{J,10\}$ where $J$ is the constant appearing in the statement of \cref{prop:nibbling-count}. Our proof consists of two stages. In Stage 1, we describe our algorithm for constructing $1$-factorizations. In Stage 2, we analyze this algorithm and show that it actually outputs the `correct' number of distinct $1$-factorizations.\\ 

{\bf Stage 1:} Our algorithm consists of the following five steps.
\begin{enumerate} [$\text{Step }1$]
\item Let $\varepsilon = n^{-1/C}$ and $p = \varepsilon^{2}$. Since $\varepsilon^{4} = n^{-4/C} \gg \log{n}/n$, $p = \varepsilon^{2} = \omega (\log{n}/\varepsilon^{3}n) \leq 1$, and $d\geq n/2 + \varepsilon n$, it follows from \cref{lemma: G contains a resilient subgraph} that there exists a spanning subgraph $H$ of $G$ which is $(\varepsilon/10, r_1, n/2)$-good, for some $r_1 = \Theta(dp)$. Later (in Step 5.), we will apply \cref{proposition-completion} with $H$ as the underlying good graph. For this, note that as required for this proposition, we indeed have that $n^{1/10} \ll r_1 = \Theta(dp) = \Theta(n\varepsilon^{2}) \ll n$  and $\log{n} \ll r_1 \alpha^{2} = \Theta(n\varepsilon^{4}) \ll n$.   


\item Let $G'$ be the graph obtained from $G$ by deleting all the edges of $H$. Then, $G'$ is $d':=(d-r_1)$-regular and crucially, $d'\geq n/2 + 3\varepsilon n/4$, since $r_1 = \Theta(dp) = \Theta(n\varepsilon^{2}) \ll \varepsilon n$. For $K=\lfloor \varepsilon^{-10} \rfloor$, fix any collection of $K^3$ subsets of $V(G')=V(G)$, denoted by $W_1,\dots, W_{K^3}$, satisfying properties $(a),(b),(c)$ and $(d)$ as in the proof of \cref{partition lemma} (which is applicable since $K\in [\log^{2}{n},n^{1/300}]$ and $d' \geq n/2$). For $i\in [K^3]$, let $U_i:= V(G')\setminus W_i$. 

\item Let $H_1,\dots,H_{K^3}$ be edge-disjoint spanning subgraphs of $G'$ satisfying properties $1., 2., 3.$ and $4.$ in the conclusion of \cref{partition lemma} for the choice of $\{U_i,W_i\}_{i=1}^{K^3}$ as above, and with $\tau = 200/K$. In particular, by conclusion $3.$ of \cref{partition lemma}, we have 
\begin{equation}
\label{eqn:min-deg-hiwi}
\delta(H_i[W_i]) \geq \left(\frac{d'}{n} - \tau\right)|W_i| \geq \left(\frac{1}{2} + \frac{3\varepsilon}{4} - \tau\right)|W_i| \geq \left(\frac{1}{2} + \frac{\varepsilon}{2}\right)|W_i| 
\end{equation}
for all $i\in [K^3]$, where the last inequality holds since $\tau = \Theta(K^{-1}) = \Theta(\varepsilon^{10}) \ll \varepsilon$. Moreover, by conclusion $4.$ of \cref{partition lemma}, we also have that for all $i\in [K^3]$,
$$\frac{d'}{K^3} \geq \delta(H_i[U_i]) \geq (1-\tau)\frac{d'}{K^3} = \left(1-\frac{200}{K}\right)\frac{d'}{K^3}.$$ 


\item Note that for each $i\in [K^3]$, we can apply \cref{prop:nibbling-count} to the graph $H_i[U_i]$ since $\delta_i := \delta(H_i[U_i]) \geq (1-200/K)d'/K^{3} = \Theta(\varepsilon^{30} n) \gg n^{1/10} \geq |U_i|^{1/10}$ and $\Delta(H_i[U_i]) - \delta(H_i[U_i]) \leq (2\Delta(H_i[U_i]))^{4/5} \ll  (\Delta(H_i[U_i]))^{5/6}$, where the first inequality follows from conclusion $4.$ of \cref{partition lemma}. Let $\mathcal M_i$ denote a collection of matchings of $H_i[U_i]$ satisfying the conclusions of \cref{prop:nibbling-count}.  In particular, each matching $M \in \mathcal{M}_i$ covers at least 
$$|U_i| - |U_i|^{1-1/10J} \geq |U_i| - \frac{n}{n^{1/10J}} \geq |U_i| - \frac{n}{K^{10}}$$ 
vertices, where the last inequality holds since $K^{10} = \Theta(\varepsilon^{-100}) \ll n^{1/10J}$, and each vertex $u \in U_i$ is uncovered by at most $$|U_i|^{-1/10J}\delta_i \leq \frac{n}{K^3n^{1/10J}} \leq \frac{n}{K^{10}} $$ 
matchings in $\mathcal{M}_i$, where the last inequality holds since $K^{7} = \Theta(\varepsilon^{-70}) \ll n^{1/10J}$.  

Note also that for each $i\in [K^3]$, we can apply \cref{lemma:extend-to-pm} to $H:=H_i$ and $\mathcal{M}:=\mathcal{M}_i$ with $\tau = \varepsilon$. Indeed, hypotheses $1.$ and $2.$ follow from conclusion $1.$ of \cref{partition lemma}, hypothesis $3.$ follows from \cref{eqn:min-deg-hiwi}, hypothesis $4.$ follows from conclusion $3.$ of \cref{partition lemma} (since $|W_i|/10K^{3} = \Theta(n/K^{5}) \gg n/K^{6}$), hypotheses $(a)$ and $(b)$ follow from the computations earlier in this step, and finally, the number of matchings in $\mathcal{M}_i$ is at most $\delta_i \leq n/K^{3}$.   
Hence, we can extend $\mathcal M_i$ to a collection of edge-disjoint perfect matchings of $H_i$, which we will denote by $\overline{\mathcal M_i}$. 


\item 
Let $R$ be the graph consisting of all the edges in $E(G')$ which do not belong to any $\overline{\mathcal M_i}$. Then, $R$ is an $r_2$-regular graph with 
$$r_2 = d'-\sum_{i=1}^{K^3}\delta_{i} \leq d' - d'(1-200/K) = 200d'/K.$$
Since $H$ is a $(\varepsilon/10,r_1,n/2)$-good graph (with $\alpha:=\varepsilon/10$ and $r_1$ satisfying the hypotheses of \cref{proposition-completion}), and since $r_2 = 200d'/K = \Theta(n\varepsilon^{10}) \ll \alpha^{4}r_1/\log{n} = \Theta(n\varepsilon^{6}/\log{n})$, we can apply \cref{proposition-completion} to the graph $H\cup R$ in order to complete $\cup_{i\in [K^3]}\overline{\mathcal M_i}$ to obtain a $1$-factorization of $G$. 
\end{enumerate}

{\bf Stage 2:} We now show that the above algorithm can output the `correct' number of distinct $1$-factorizations. 
Throughout, we will assume that $n$ is a sufficiently large even integer. By \cref{rmk:counting partitions}, there are at least 
\begin{align*}
\left(1-2n^{-8}\right)(K^{3})^{\left(1-\frac{10}{K}\right)\frac{nd'}{2}} & =\left(1-2n^{-8}\right)\left(K^{3}e^{-30\log K/K}\right)^{nd'/2}\\
 & \gg \left(e^{-30\log{K}/K}\right)^{nd'/2}\left(K^{3}e^{-30\log K/K}\right)^{nd'/2}\\\
 & \geq\left(K^{3}e^{-60\log K/K}\right)^{nd'/2}\\
 & \geq\left(K^{3}e^{-60\log K/K}\right)^{\frac{nd}{2}(1-p)}\\
 & = \left(K^{3}e^{-60\log K/K}\right)^{nd/2}\left(K^{-3p}e^{60p\log K/K}\right)^{nd/2} \\ 
 & \gg \left(K^{3}e^{-60\varepsilon^{10}\log K}\right)^{\frac{nd}{2}}\left(e^{-3p\log K}\right)^{\frac{nd}{2}}\\
 & \geq \left(K^{3}\right)^{\frac{nd}{2}}\left(e^{-6p\log{K}}\right)^{\frac{nd}{2}}
\end{align*}
distinct ways to choose the collection of subgraphs $H_1,\dots,H_{K^3}$ in Step 3. 
Moreover, by \cref{prop:nibbling-count}, for each $i\in [K^3]$, there are at least
\begin{align*}
\left(\frac{(1-|U_i|^{-1/40J})\delta_i}{e^2}\right)^{\frac{ \delta_i |U_i|}{2}} 
& \geq \left(\frac{(1-2n^{-1/40J})\delta_i}{e^2}\right)^{\frac{ \delta_i n(1-2/K^{2})}{2}} \\
& \geq \left(\frac{(1-2n^{-1/40J})(1-200/K)d'}{K^{3}e^2}\right)^{\frac{nd'(1-200/K)(1-2/K^{2})}{2K^3}} \\
& \geq \left(\frac{(1-400\varepsilon^{10})d'}{K^{3}e^2}\right)^{\frac{nd'(1-400\varepsilon^{10})}{2K^{3}}}\\ 
& \geq \left(\frac{(1-400\varepsilon^{10})(1-p)d}{K^{3}e^2}\right)^{\frac{nd(1-p)(1-400\varepsilon^{10})}{2K^{3}}}\\
 & \geq\left(\frac{(1-2p)d}{K^{3}e^{2}}\right)^{\frac{nd}{2K^{3}}\left(1-2p\right)}\\
 & \geq\left(\frac{(1-2p)d}{K^{3}e^{2}}\right)^{\frac{nd}{2K^{3}}}\left(e^{-2p\log d}\right)^{\frac{nd}{2K^{3}}}
\end{align*}
distinct ways to choose a collection of edge-disjoint matchings $\mathcal{M}_i$ in $H_i[U_i]$ satisfying the conclusions of \cref{prop:nibbling-count}. Since distinct collections of edge-disjoint matchings $\mathcal{M}_i$ of $H_i[U_i]$ stay distinct upon the application of \cref{lemma:extend-to-pm} (see \cref{rmk:distinct-stays-distinct}), it follows that there are at least 
$$\prod_{i=1}^{K^3}\left(\frac{(1-2p)d}{K^{3}e^{2}}\right)^{\frac{nd}{2K^{3}}}\left(e^{-2p\log d}\right)^{\frac{nd}{2K^{3}}} = \left(\frac{(1-2p)d}{K^3 e^2}\right)^{\frac{nd}{2}}\left(e^{-2p\log{d}}\right)^{\frac{nd}{2}}$$
distinct ways to choose, one for each $i\in [K^3]$, a collection of edge-disjoint perfect matchings $\mathcal{M}_i$ of $H_i$ as at the conclusion of Step 4. of our algorithm in Stage 1. Together with the number of choices for $H_1,\dots,H_{K^3}$ in Step 3., it follows that the \emph{multiset} of $1$-factorizations of $G$ that can be obtained by the algorithm in Stage 1 has size at least  
\[
\left(\left(\frac{(1-2p)d}{K^{3}e^{2}}\right)K^{3}e^{-10p\log n}\right)^{\frac{nd}{2}}\geq\left(\frac{(1-20p\log{n})d}{e^{2}}\right)^{\frac{nd}{2}}.
\]
To complete the proof, it suffices to show that no $1$-factorization $\mathcal{F} = \{F_1,\dots,F_d\}$ is counted more than $\left(1+400p\log{n}\right)^{\frac{nd}{2}}$ times in the calculation above. Let us call a collection of edge-disjoint subgraphs $H_1,\dots,H_{K^3}$ of $G'$ 
\emph{consistent} with $\mathcal{F}$ if $H_1,\dots,H_{K^3}$ satisfy the conclusions of \cref{partition lemma} with the parameters in Step 3 (in particular, with respect to the \emph{fixed} collection $\{U_i,W_i\}_{i=1}^{K^3}$ from Step 2), and if $\mathcal{F}$ can be obtained by the algorithm after choosing $H_1,\dots,H_{K^3}$ in Step 3. It is clear that the number of times that $\mathcal{F}$ can be counted by the above computation is at most the number of collections $H_1,\dots,H_{K^3}$ which are consistent with $\mathcal{F}$, so that it suffices to upper bound the latter. For this, note that at most $r_1+200d'/K \leq 2r_1$ of the perfect matchings in $\mathcal{F}$ can come from Step 5 of the algorithm. Each of the other perfect matchings belongs completely to a single $H_i$ by construction. 
It follows that the number of consistent collections $H_1,\dots,H_{K^3}$ is at most 
$$n\cdot{d \choose 2r_{1}}\cdot\left(K^{3}\right)^{r_1 n}\cdot\left(K^{3}\right)^{d}.$$ 
Indeed, there are at most $n\cdot{d \choose 2r_{1}}$ ways to choose the perfect matchings coming from Step 5; these matchings contain at most $r_1 n$ edges; for each such edge, there are at most $K^3$ choices for which $H_i$ it should belong to; and for each of the remaining (at most $d$) matchings which are completely contained in some $H_i$, there are at most $K^{3}$ choices for which $H_i$ such a matching should belong to. Finally, observe that  
$$n\cdot {d \choose 2r_{1}}\cdot \left(K^{3}\right)^{r_1n}\cdot (K^3)^{d}  \ll d^{3r_{1}}K^{3r_{1}n + 3d} \ll K^{4r_1 n} \ll K^{10pnd} = \left(K^{20p}\right)^{\frac{nd}{2}}
\leq\left(1+400p\log{n}\right)^{\frac{nd}{2}},$$
which completes the proof. 
\end{proof}

\section{Proof of \cref{prop:nibbling-count}}
\label{appendix:nibbling}
We now show how the proof of the main result in \cite{DGP}, which is based on the celebrated R\"odl nibble \cite{Rodl}, implies \cref{prop:nibbling-count}. The organization of this section is as follows: \cref{algorithm:nibble} records the nibbling algorithm used in \cite{DGP}; \cref{thm:nibbling-analysis} records the conclusion of the analysis in \cite{DGP}; \cref{prop:good-nibbles} adapts \cref{thm:nibbling-analysis} for our choice of parameters; \cref{rmk:nibble-large-matchings} shows that the collection of matchings produced by \cref{algorithm:nibble} satisfies the conclusions of \cref{prop:nibbling-count}, and finally,
following this remark, we present the proof of \cref{prop:nibbling-count}. \\

The following algorithm (\cref{algorithm:nibble}) is a slight variant of the algorithm used in \cite{DGP} to find an almost-optimal edge coloring of a graph. Here, we show how it can be used to generate `almost the correct number' of `equitable collections of edge-disjoint large matchings' of an `almost regular graph' (all in the sense of \cref{prop:nibbling-count}).  Since we are not concerned with the running time of the algorithm, we are able to make a simpler choice for the initial palettes of the edges as compared to \cite{DGP}. Moreover, since our goal is to output a large collection of edge-disjoint matchings as in \cref{prop:nibbling-count}, we have no need for the trivial `Phase 2' of the algorithm in \cite{DGP}. \\
\begin{algorithm}
\caption{The Nibble Algorithm}
\label{algorithm:nibble}
Input: The initial graph $G_0:= G$ on $n$ vertices with minimum degree $\delta$ and maximum degree $\Delta$. 
Each edge $e = uv$ is initially given the palette $A_0(e) = \{1,\dots, \delta\}$.
For $i=0,1,\dots,t_\tau - 1$ stages, repeat the following:
\begin{itemize}
\item \emph{(Select nibble)} Each vertex $u$ randomly and independently selects each uncolored edge incident to itself with probability $\tau/2$. An edge is considered selected if either or both of its endpoints selects it. 
\item \emph{(Choose tentative color)} Each selected edge $e$ chooses independently at random a tentative color $t(e)$ from its palette $A_i(e)$ of currently available colors. 
\item \emph{(Check color conflicts)} Color $t(e)$ becomes the final color of $e$ unless some edge incident to $e$ has chosen the same tentative color. 
\item \emph{(Update graph and palettes)} The graph and the palettes are updated by setting 
$$ G_{i+1} = G_i - \{e \vert e \text{ got a final color}\}$$
and, for each edge $e$, setting 
$$ A_{i+1}(e) = A_i(e) - \{t(f)\vert f \text{ incident to } e, t(f) \text{ is the final color of } f \}.$$
\end{itemize}
\end{algorithm}

The analysis of this algorithm is based on controlling the following three quantities:
\begin{itemize}
\item $|A_i(u)|$, the size of the implicit palette of vertex $u$ at the end of stage $i$, where the implicit palette $A_i(u)$ denotes the set of colors not yet successfully used by any edge incident to $u$. 
\item $|A_i(e)|$, the size of the palette $A_i(e)$ of edge $e$ at the end of stage $i$. Note that $A_i(uv) = A_i(u) \cap A_i(v)$. 
\item $\deg_{i,\gamma}(u)$, the number of neighbors of $u$ which, at the end of stage $i$, have color $\gamma$ in their palettes. 
\end{itemize}

We record the outcome of their analysis as \cref{thm:nibbling-analysis}. Before stating it, we need some notation.  

Define $d_i$ and $a_i$ as follows. First, define initial values
$$d_0, a_0 := \Delta$$
and then, recursively define 
\begin{align*}
d_{i} & :=(1-p_{\tau})d_{i-1}=(1-p_{\tau})^{i}\Delta\\
a_{i} & :=(1-p_{\tau})^{2}a_{i-1}=(1-p_{\tau})^{2i}\Delta=d_{i}^{2}/\Delta,
\end{align*}
where 
\[
p_{\tau}:=\tau\left(1-\frac{\tau}{4}\right)e^{-2\tau(1-\tau/4)}.
\]
In particular, note that setting 
$$t_\tau := \frac{1}{p_\tau}\log{\frac{4}{\tau}},$$ 
we have $d_{t_\tau} \leq {\tau\Delta}/4$. 
\begin{theorem}[\cite{DGP}, Lemmas 10, 13 and 16, and the discussion in Section 5.5.]
\label{thm:nibbling-analysis}
There exist constants $K,c> 0$ such that, if at the end of stage $i$ of \cref{algorithm:nibble}, the following holds for all vertices $u$, edges $e$ and colors $\gamma$: 
\begin{align*}
|A_{i}(u)| & =(1\pm e_{i})d_{i}\\
|A_{i}(e)| & =(1\pm e_{i})a_{i}\\
\deg_{i,\gamma}(u) & =(1\pm e_{i})a_{i}
\end{align*}
then, except with probability at most $15n^{-1}$, the following holds at the end of stage $i+1$ for all vertices $u$, edges $e$ and colors $\gamma$:
\begin{align*}
|A_{i+1}(u)| & =(1\pm e_{i+1})d_{i+1}\\
|A_{i+1}(e)| & =(1\pm e_{i+1})a_{i+1}\\
\deg_{i+1,\gamma}(u) & =(1\pm e_{i+1})a_{i+1},
\end{align*}
where
\[
e_{i+1}=C(e_{i}+c\sqrt{\log n/a_{i}})=C(e_{i}+c(1-p_{\tau})^{-i}\sqrt{\log n/\Delta}),
\]
with $C = 1+K\tau$. 
\end{theorem}

\begin{remark}In our case, we have $$|A_0(u)| =|A_0(e)| = \delta = \Delta\left(1 - \frac{\Delta - \delta}{\Delta}\right),$$ 
and 
$$\deg_{0,\gamma} = \deg(u) = \Delta\left(1 \pm \frac{\Delta - \delta}{\Delta}\right)$$
for all vertices $u$, edges $e$, and colors $\gamma$. Therefore, we can take $$ e_0 = \frac{\Delta - \delta}{\Delta}\leq \Delta^{-1/6}.$$
\end{remark}

\begin{proposition} 
\label{prop:good-nibbles}
Let $n$ be a sufficiently large integer. There exists a constant $J > 0$ for which the following holds. Let $G$ be a graph on $n$ vertices with $\Delta \geq \log^{12}n$ and $e_0: = (\Delta - \delta)/\Delta \leq \Delta^{-1/6}$, and let $\Delta^{-1/J}\leq \tau < 1/100$.
Then, with probability at least $1 - 15t_\tau n^{-1} $, the following holds for the execution of \cref{algorithm:nibble} on $G$ with parameter $\tau$ for $t_{\tau}-1$ stages: for all $0 \leq i \leq t_\tau$, and for all vertices $u$, all edges $e$, and all colors $\gamma$,
\begin{itemize}
\item $|A_i(u)| = (1\pm \tau^{3}) d_i$
\item $|A_i(e)| = (1 \pm \tau^{3}) a_i$
\item $\deg_{i,\gamma}(u) = (1\pm \tau^{3}) a_i$
\end{itemize}
\end{proposition}

\begin{proof}
Setting $A:= c\sqrt{\log{n}/\Delta}$ and $B:=1/(1-p_\tau)$, we see from \cref{thm:nibbling-analysis} and the union bound that, except with probability at most $15t_{\tau}n^{-1}$, 
$$e_\ell = C^{\ell}e_0 + A[C^\ell + C^{\ell-1}B+\dots + CB^{\ell-1}]$$
for all $0\leq \ell \leq t_{\tau}$. Since $B = 1/(1-p_\tau) \leq 1 + K'\tau$ for some constant $K' > 0$, it follows that 
$$e_\ell \leq \ell(1+L\tau)^\ell c \sqrt{\log{n}/\Delta} + (1+L\tau)^\ell e_0,$$
where $L = \max\{K,K',c\}$. Since $e_0 \leq \Delta^{-1/6}$ and $\Delta \geq \log^{12}{n}$, it follows that 
$$ e_\ell \leq 2c\ell \exp(L\tau \ell) \Delta^{-1/6}.$$
The right hand side is maximized when $\ell = t_\tau$. Finally, since $\tau < 1/100$ by assumption, we get that  
$$e_{t_\tau}\leq \left(\frac{1}{\tau}\right)^{3L}\Delta^{-1/6} \leq \tau^{3},$$
where the last inequality holds provided we take $J \geq 18(L+1)$.
\end{proof}

\begin{remark}  
\label{rmk:nibble-large-matchings}
Consider any partial edge coloring of $G$ satisfying the conclusions of \cref{prop:good-nibbles}, and let $\mathcal{M}:=\{M_1,\dots,M_\delta\}$ denote the collection of edge-disjoint matchings of $G$ obtained by letting $M_{\gamma}$ be the set of edges colored with $\gamma$. Then, $|M_{\gamma}| \geq (1-\tau)n/2 $ for all $\gamma \in [\delta]$. To see this, note that any vertex $u$ which is not covered by $M_\gamma$ must have $\gamma$ in its implicit palette $A_{t_\tau}(u)$. Moreover, every vertex $u$ has at least $|A_{t_\tau}(u)| > d_{t_\tau}/2$ missing colors, and therefore at least as many uncolored edges attached to it at the end of stage $t_\tau - 1$. It follows that every vertex which is uncovered by $M_\gamma$ contributes at least $d_{t_\tau}/2$ to the sum $\sum_{u}\deg_{t_\tau, \gamma}(u)$. Hence, if $n_\gamma$ denotes the number of vertices uncovered by $M_\gamma$, then 
$$ \frac{n_\gamma d_{t_\tau}}{2} \leq \sum_{u}\deg_{t_\tau, \gamma}(u).$$
On the other hand, we have 
$$ \sum_{u}\deg_{t_\tau, \gamma}(u) \leq 2a_{t_\tau}n \leq \frac{2d_{t_\tau}^2 n}{\Delta}.$$
Combining these two inequalities and using $d_{t_\tau} \leq \tau \Delta/4$, we see that $n_\gamma \leq \tau n$, as desired. 

Moreover, as mentioned above, the number of times a given vertex $u$ is left uncovered by a matching in $\mathcal{M}$ equals $|A_{t_\tau}(u)|$, which is at most $2d_{t_\tau} \leq \tau \Delta/2$. Below, in the proof of \cref{prop:nibbling-count}, we will take $\tau := \Delta^{-1/J}$. Using this choice of parameter in the two estimates in the present remark, it follows readily that that $\mathcal{M}$ satisfies the conclusions of \cref{prop:nibbling-count}. 

\end{remark}

We are now ready to prove \cref{prop:nibbling-count}. 
\begin{proof}[Proof of \cref{prop:nibbling-count}]
\label{proof-nibbling-count}
We will view the execution of \cref{algorithm:nibble} as a branching process where at each stage, we branch out according to which edges get assigned final colors, and which final colors are assigned to these edges. In particular, the leaves of this branching process are at distance $t_\tau -1$ from the root.

We say that a leaf $L$ of this branching process is a \emph{good leaf} if the unique path from the root to $L$ represents an execution of \cref{algorithm:nibble} such that at all stages, all vertices $u$, all edges $e$, and all colors $\gamma$ satisfy the conclusions of \cref{prop:good-nibbles}. Also, we say that a partial coloring with $\delta$ colors is a \emph{good partial coloring} if the corresponding collection of edge-disjoint matchings $\mathcal{M} = \{M_{1},\dots,M_{\delta}\}$ satisfies the conclusions of \cref{prop:nibbling-count}. Note by \cref{rmk:nibble-large-matchings} that the partial edge coloring corresponding to a good leaf is a good partial coloring. Therefore, in order to lower bound the number of good partial colorings, it suffices to lower bound the number of good leaves, and upper bound the number of distinct leaves any given good partial coloring can correspond to. 

For this, let $Q > 0$ be an upper bound for the probability of the branching process reaching a given good leaf. Since, by \cref{prop:good-nibbles}, the probability that the branching process reaches some good leaf is at least $1-15t_\tau n^{-1}$, it follows that the number of good leaves is at least $(1-15t_\tau n^{-1})/Q$. Further, let $R > 0$ be such that for any good partial coloring $\mathcal{C}$, there are at most $R$ leaves of the branching process whose corresponding partial coloring is $\mathcal{C}$. Then, it follows that the number of good partial colorings is at least $(1-15t_\tau n^{-1})/QR$. The remainder of the proof consists of upper bounding the quantity $QR$. 

To upper bound $Q$, fix a good leaf $L$ and note that in the $i^{th}$ stage of the execution corresponding to $L$, 
$m_i$ specific edges must be selected and assigned their final colors, where 
\begin{align*}
m_{i}: & =(1\pm\tau^{3})\frac{d_{i}n}{2}-(1\pm\tau^{3})\frac{d_{i+1}n}{2}
 =\frac{d_{i}n}{2}\left(1\pm\tau^{3}-(1-p_{\tau})(1\pm\tau^{3})\right)\\
 &= \frac{d_{i}n}{2}\left(p_\tau \pm 3\tau^{3} \right)
  =\frac{\tau d_{i}n}{2}\left(1\pm10\tau\right),
\end{align*}
and the last line follows since $\tau < 1/100$. 
Since each edge is selected independently with probability $\tau(1-\tau/4)$ (an edge is selected if and only if it is selected by at least one of its endpoints, which happens with probability $\tau/2 + \tau/2 - \tau^{2}/4$ by the inclusion-exclusion principle), and each selected edge chooses one of at least $(1-\tau^{3})a_i$ colors uniformly at random, it follows that the probability that the branching process makes the specific choices at the $i^{th}$ stage of $L$ is at most 
\begin{align*}
\left(\frac{\tau(1-\frac{\tau}{4})}{(1-\tau^{3})a_{i}}\right)^{m_{i}}\left(1-\tau\left(1-\frac{\tau}{4}\right)\right)^{\frac{(1-\tau^{3})d_{i}n}{2}-m_{i}} & \leq\left(\frac{\tau(1+\tau)}{a_{i}}\right)^{m_{i}}\left(1-\tau\left(1-\frac{\tau}{4}\right)\right)^{\frac{(1-\tau^{3})d_{i}n}{2}-m_{i}}\\
 & \leq\left(\frac{\tau(1+\tau)}{a_{i}}\right)^{m_{i}}\left(1-\tau\left(1-\frac{\tau}{4}\right)\right)^{\frac{(1-\tau^{3})d_{i}n-(1\pm10\tau)\tau d_{i}n}{2}}\\
 & \leq\left(\frac{\tau(1+\tau)}{a_{i}}\right)^{m_{i}}\exp\left(-\tau\left(1-\frac{\tau}{4}\right)\frac{d_{i}n}{2}\right)^{(1-\tau^{3})-(1\pm10\tau)\tau}\\
 & \leq\left(\frac{\tau(1+\tau)}{a_{i}}\right)^{m_{i}}\exp\left(-\left(1-10\tau\right)\frac{\tau d_{i}n}{2}\right),
\end{align*}
where the first and last inequalities use $\tau < 1/100$. 
Since the randomness in different stages of the branching process is independent, it follows that 
\begin{eqnarray*}
Q & \leq&\prod_{i=0}^{t_{\tau}-1}\left(\frac{\tau(1+\tau)}{a_{i}}\right)^{m_{i}}\exp\left(-(1-10\tau)\frac{\tau d_{i}n}{2}\right)
 =\exp\left(-(1-10\tau)\frac{\tau n}{2}\sum_{i=0}^{t_{\tau}-1}d_{i}\right)\prod_{i=0}^{t_{\tau}-1}\left(\frac{\tau(1+\tau)}{a_{i}}\right)^{m_{i}}\\
 & =&\exp\left(-(1-10\tau)(1\pm10\tau)\frac{\Delta n}{2}\right)\prod_{i=0}^{t_{\tau}-1}\left(\frac{\tau(1+\tau)}{a_{i}}\right)^{m_{i}}
 =\exp\left(-(1\pm25\tau)\frac{\Delta n}{2}\right)\prod_{i=0}^{t_{\tau}-1}\left(\frac{\tau(1+\tau)}{\Delta(1-p_{\tau})^{2i}}\right)^{m_{i}}\\
 & =&\exp\left(-(1\pm25\tau)\frac{\Delta n}{2}\right)\left(\frac{\tau(1+\tau)}{\Delta}\right)^{\sum_{i=0}^{t_{\tau}-1}m_{i}}(1-p_{\tau})^{-\sum_{i=0}^{t_{\tau}-1}2im_{i}}\\
 & =&\exp\left(-(1\pm25\tau)\frac{\Delta n}{2}\right)\left(\frac{\tau(1+\tau)}{\Delta}\right)^{\frac{(1\pm10\tau)\tau n}{2}\sum_{i=0}^{t_{\tau}-1}d_{i}}(1-p_{\tau})^{-\sum_{i=0}^{t_{\tau}-1}2im_{i}}\\
 & =&\exp\left(-(1\pm25\tau)\frac{\Delta n}{2}\right)\left(\frac{\tau(1+\tau)}{\Delta}\right)^{\frac{(1\pm25\tau)\Delta n}{2}}(1-p_{\tau})^{-\sum_{i=0}^{t_{\tau}-1}2im_{i}}\\
 &=&\exp\left(-(1\pm30\tau)m\right)\left(\frac{\tau(1+\tau)}{\Delta}\right)^{(1\pm30\tau)m}(1-p_{\tau})^{-\sum_{i=0}^{t_{\tau}-1}2im_{i}}, 
\end{eqnarray*}
where we have used that 
$$\sum_{i=0}^{t_{\tau}-1}d_{i}  =\Delta\sum_{i=0}^{t_{\tau}-1}(1-p_{\tau})^{i}\\
  =(1\pm10\tau)\frac{\Delta}{\tau} $$
in the third and seventh lines, $\tau < 1/100$ in the fourth and seventh lines, and $ m = (1\pm \tau)\Delta n/2$ in the last line. 

To upper bound $R$, it suffices to upper bound the number of ways in which the edges of any good partial coloring can be partitioned into sets of size $\{(1\pm 10\tau)\tau d_i n/2\}_{i=0}^{t_\tau - 1}$, since the number of edges colored by the algorithm in the $i^{th}$ stage is $(1\pm 10\tau)\tau d_i n/2$. For this, note that there are at most $(10\tau^{2}\Delta n)^{t_\tau}$ ways to choose the sizes of these $t_\tau$ sets, and for each such choice for the sizes of the sets, there are at most $m!/\prod_{i=0}^{t_\tau -1}((1-10\tau)m_i)!$ ways to partition the edges into these sets. Therefore, 
\begin{align*}
R(10\tau^{2}\Delta n)^{-{t_\tau}} & \leq\frac{m!}{\prod_{i=0}^{t_{\tau}-1}((1-10\tau)m_{i})!} \leq m\left(\frac{m}{e}\right)^{m}\prod_{i=0}^{t_{\tau}-1}\left(\frac{e}{(1-10\tau)m_{i}}\right)^{(1-10\tau)m_{i}}\\
 & =m \left(\frac{m}{e}\right)^{m}\prod_{i=0}^{t_{\tau}-1}\left(\frac{2e}{(1-10\tau)(1\pm10\tau)\tau d_{i}n}\right)^{(1-10\tau)m_{i}}\\
 & =m \left(\frac{m}{e}\right)^{m}\prod_{i=0}^{t_{\tau}-1}\left(\frac{2e}{(1-25\tau)\tau\Delta(1-p_{\tau})^{i}n}\right)^{(1-10\tau)m_{i}}\\
 & =m \left(\frac{m}{e}\right)^{m}\left(\frac{2e}{(1-25\tau)\tau\Delta n}\right)^{(1-10\tau)\sum_{i=0}^{t_{\tau}-1}m_{i}}\prod_{i=0}^{t_{\tau}-1}(1-p_{\tau})^{-im_{i}(1-10\tau)}\\
 & =m \left(\frac{m}{e}\right)^{m}\left(\frac{2e}{(1-25\tau)\tau\Delta n}\right)^{\frac{(1\pm25\tau)\Delta n}{2}}(1-p_{\tau})^{-(1-10\tau)\sum_{i}im_{i}}\\
 & \leq m \left(\frac{m}{e}\right)^{m}\left(\frac{e}{(1-25\tau)\tau m}\right)^{(1\pm30\tau)m}(1-p_{\tau})^{-(1-10\tau)\sum_{i}im_{i}}\\
 & \leq\left(\frac{m}{e}\right)^{30\tau m}\left(\frac{(1\pm50\tau)}{\tau}\right)^{m(1\pm30\tau)}(1-p_{\tau})^{-(1-10\tau)\sum_{i}im_{i}}\\
 & \leq\left(me\right)^{100\tau m}\left(\frac{1}{\tau}\right)^{m(1\pm30\tau)}(1-p_{\tau})^{-(1-10\tau)\sum_{i}im_{i}}, 
\end{align*}
where the second line uses the standard approximation $ (k/e)^k \leq k! \leq k(k/e)^k$ for $k \geq 10$; the third line uses the definition of $m_i$; the fourth line uses the definition of $d_i$ along with $\tau < 1/100$; the sixth line uses $\sum_{i=0}^{t_\tau - 1}m_i = (1\pm 25)\Delta n/2$ as shown in the calculation of the upper bound on $Q$; the seventh line uses $m = (1\pm \tau)\Delta n/2$; the eighth line uses $ \Delta^{-1/J} < \tau < 1/100$, and the last line uses $(1+50\tau) \leq e^{50\tau}$ and $\tau < 1/100$.    

It follows that 
\begin{align*}
QR & \leq(10\tau^{2}\Delta n)^{t_\tau}\left(me\right)^{100\tau m}\exp\left(-(1-30\tau)m\right)\left(\frac{\tau(1+\tau)}{\Delta}\right)^{(1\pm30\tau)m}\left(\frac{1}{\tau}\right)^{m(1\pm30\tau)}(1-p_{\tau})^{-\sum_{i}3im_{i}}\\
 & \leq(10\tau^{2}\Delta n)^{t_\tau}\left(\frac{me\Delta}{\tau}\right)^{200\tau m}\left(\frac{1}{e\Delta}\right)^{m}(1-p_{\tau})^{-\sum_{i}3im_{i}}\\
 & \leq(10\tau^{2}\Delta n)^{t_\tau}\left(\frac{me\Delta}{\tau}\right)^{200\tau m}\left(\frac{1}{e\delta}\right)^{m}e^{3p_{\tau}(1+10\tau)\sum_{i=0}^{t_{\tau}-1}im_{i}}\\
 & \leq m^{t_\tau}(e\delta)^{-m}\left(\frac{me\Delta}{\tau}\right)^{200\tau m}e^{3(1\pm30\tau)\frac{\tau\Delta n}{2p_{\tau}}}\\
 & \leq m^{\tau m}(e\delta)^{-m}\left(\frac{me\Delta}{\tau}\right)^{200\tau m}e^{3(1\pm50\tau)m}\\
 & \leq (e\delta)^{-m}e^{3m}\left(\frac{me\Delta}{\tau}\right)^{1000\tau m}\\
 & =\left(\frac{e^{2}}{\delta}\right)^{m}e^{1000\tau m\log(me\Delta/\tau)}\\
 & \leq\left(\frac{e^{3000\tau\log(m/\tau)}e^{2}}{\delta}\right)^{m},
\end{align*}
where the second line uses the inequality $(1+\tau) \leq e^{\tau}$; the third line uses the inequality $(1-p_\tau) \geq e^{-p_\tau(1+10\tau)}$ which holds since $\tau < 1/100$; the fourth line uses the following computation:
\begin{align*}
\sum_{i=0}^{t_{\tau}-1}im_{i} & =(1\pm10\tau)\frac{\tau n}{2}\sum_{i=0}^{t_{\tau}-1}id_{i}
  =(1\pm10\tau)\frac{\tau\Delta n}{2}\sum_{i=0}^{t_{\tau}-1}i(1-p_{\tau})^{i}\\
 & \leq(1\pm10\tau)\frac{\tau\Delta n}{2}\sum_{i=0}^{\infty}i(1-p_{\tau})^{i}
 =(1\pm10\tau)\frac{\tau\Delta n}{2}\frac{1}{p_{\tau}^{2}},
\end{align*}
along with $\tau < 1/100$, and the fifth line uses $m=(1\pm \tau)\Delta n/2$ along with $\tau < (1+5\tau)p_\tau$ and $t_\tau < 1/\tau^{2} < \tau m$ which holds since $m^{-1/J} \leq \tau < 1/100$.  
Finally, substituting $\tau = \Delta^{-1/J}$ completes the proof. 
\end{proof}

\section{Concluding remarks and open problems}

\begin{itemize}
\item We proved that the number of $1$-factorizations in a $d$-regular graph is at least 
$$\left((1+o(1))\frac{d}{e^2}\right)^{dn/2},$$ 
provided that $d\geq \frac{n}{2}+\varepsilon n$. As mentioned in the introduction, this is asymptotically best possible. It will be very interesting to obtain a similar result for all $d\geq 2\lceil n/4\rceil-1$ (the existence of a $1$-factorization in this regime was proven in \cite{CKLOT}).
\item As mentioned in \cref{remark:good error term}, we obtain an explicit function (polynomial in $n$) for the $(1+o(1))$-term in the bound on the number of $1$-factorizations. We have written such a formula 
with the hope that it could be useful towards studying the behavior of \emph{typical} $1$-factorizations. For example, using similar bounds on the number of Steiner triple systems as obtained by Keevash \cite{Keevash}, Kwan \cite{Kwan} was recently able to study some non-trivial properties of typical Steiner Triple Systems. Therefore, we hope that building upon Kwan's ideas and using our counting argument one could obtain some non-trivial properties of typical $1$-factorizations. For example, can one show that a typical $1$-factorization of $K_n$ contains a \emph{rainbow} Hamiltonian path? (that is, a Hamiltonian path which uses exactly one edge from each of the perfect matchings). 
\item Another very interesting direction is to study the number of $1$-factorizaitons in hypergraphs. In this setting, much less is known and every non-trivial lower bound on the number of such factorizations should require new ideas. We are curious whether one can attack this problem using some clever reduction to the graph setting and use our ideas for the `completion part'. 
\end{itemize}

{\bf Acknowledgment:} The first author is grateful to Kyle Luh and Rajko Nenadov for helpful discussions at the first step of this project. We would also like to thank the anonymous referees for a very thorough reading of the manuscript and numerous invaluable comments.

\end{document}